\documentclass[reqno]{amsart}
\usepackage[USenglish]{babel}
\usepackage{amssymb,latexsym}
\usepackage{amsmath}
\usepackage{amsthm}
\usepackage{amsfonts}
\usepackage{graphicx}
\usepackage[all]{xy}
\usepackage{titletoc}
\numberwithin{equation}{section}
\newtheorem{theorem}{Theorem}[section]
\newtheorem{proposition}{Proposition}[section]
\newtheorem{lemma}{Lemma}[section]

\newtheorem{remark}{Remark}[section]

\theoremstyle{definition}

\def\XXint#1#2#3{{\setbox0=\hbox{$#1{#2#3}{\int}$}
     \vcenter{\hbox{$#2#3$}}\kern-.5\wd0}}

\def\e{\varepsilon}

\def\R{{\mathbb R}}

\renewcommand{\d}{\delta }

\begin{document}

\title{Classification of blow-up limits for the sinh-Gordon equation}

\author{Aleks Jevnikar, Juncheng Wei, Wen Yang}

\address{ Aleks Jevnikar,~University of Rome `Tor Vergata', Via della Ricerca Scientifica 1, 00133 Roma, Italy}
\email{jevnikar@mat.uniroma2.it}

\address{ Juncheng ~Wei,~Department of Mathematics, University of British Columbia,
Vancouver, BC V6T 1Z2, Canada}
\email{jcwei@math.ubc.ca}

\address{ Wen ~Yang,~Center for Advanced Study in Theoretical Sciences (CASTS), National Taiwan University, Taipei 10617, Taiwan}
\email{math.yangwen@gmail.com}

\thanks{A.J. is partially supported by the PRIN project {\em Variational and perturbative aspects of nonlinear differential problems}. J.W. and W.Y. are partially supported by the NSERC of Canada. Part of the paper was written when A.J. was visiting the Mathematics Department at the University of British Columbia. He would like to thank the institute for the hospitality and for the financial support.}

\keywords{Geometric PDEs, Sinh-Gordon equation, Blow-up analysis.}

\subjclass[2000]{35J61, 35R01, 35B44.}

\begin{abstract}
The aim of this paper is to use a selection process and a careful study of the interaction of bubbling solutions to show a classification result for the blow-up values of the elliptic sinh-Gordon equation
$$
	\Delta u+h_1e^u-h_2e^{-u}=0 \qquad \mathrm{in}~B_1\subset\R^2.
$$	
In particular we get that the blow-up values are multiple of $8\pi.$ It generalizes the result of Jost, Wang, Ye and Zhou \cite{jwyz} where the extra assumption $h_1 = h_2$ is crucially used.
\end{abstract}

\maketitle

\section{Introduction}

\medskip

In this paper we mainly focus on the weak limit of the energy sequence for the following equation
\begin{equation}
\label{eq}
\Delta u+h_1e^u-h_2e^{-u}=0~\mathrm{in}~B_1\subset\R^2,
\end{equation}
where $h_1,h_2$ are smooth positive functions and $B_1$ is the unit ball in $\R^2$.
\medskip

Equation (\ref{eq}) arises in the study of the equilibrium turbulence with arbitrarily signed vortices \cite{c,mp,l,n}, and was first proposed by Onsager \cite{o}, Joyce and Montgomery \cite{jm} from different statistical arguments. When the nonlinear term $e^{-u}$ in \eqref{eq} is replaced by $\tau e^{-\gamma u}$ with $\tau,\gamma>0$, the equation \eqref{eq} stands for a more general type of equation which arises in the context of the statistical
mechanics description of 2D-turbulence. For the recent developments of such equation. We refer the readers to \cite{pr,rt,rtzz} and references therein. Moreover, it plays also a very important role in the study of the construction of constant mean curvature surfaces initiated by Wente, see \cite{jwyz,w2} and the references therein.
\medskip

When $h_2 \equiv 0$ the equation \eqref{eq} reduces to the classic Liouville equation
\begin{equation} \label{liouv}
\Delta u + h e^u=0 \qquad \mathrm{in}~B_1\subset\R^2.
\end{equation}
Equation \eqref{liouv} is important in the geometry of manifolds as it rules the change of Gaussian curvature under conformal deformation of the metric, see \cite{ba, cha, cha2, ly, s}. Another motivation for the study of \eqref{liouv} is in mathematical physics as it models the mean field of Euler flows, see \cite{clmp} and \cite{ki}. This equation has been very much studied in the literature; there are by now many results regarding existence, compactness of solutions, bubbling behavior, etc. We refer the interested reader to the reviews \cite{mal} and \cite{tar}.
\medskip

Wente's work \cite{w2} on the constant mean curvature surfaces and the work of Sacks-Uhlenbeck \cite{su} concerning harmonic maps led to investigate the blow-up phenomena for variational problems that possess a lack of compactness. Later, in a series work of Steffen \cite{s1}, Struwe \cite{s2} and Brezis, Coron \cite{bc1}, the program of understanding the blow-up for constant mean curvature surfaces was completed.
\medskip

As many geometric problems, also \eqref{eq} (and \eqref{liouv}) presents loss of compactness phenomena, as its solutions might blow-up. Concerning \eqref{liouv} it was proved in \cite{bm, li, li-sha} that for a sequence of blow-up solutions $u_k$ to \eqref{liouv} (relatively to $h^k$) with blow-up point $\bar x$ it holds
\begin{equation} \label{quant}
	\lim_{\d \to 0} \lim_{k \to \infty} \int_{B_\d (\bar x)} h^k e^{u_k} = 8\pi.
\end{equation}
Somehow, each blow-up point has a quantized local mass.
\medskip

On the other hand, the blow-up behavior of solutions of equation \eqref{eq} is not yet developed in full generality; this analysis was carried out in \cite{os, os1} and \cite{jwyz} under the assumption that $h_1=h_2$ or $h_1,h_2$ are constants. In particular, by using the deep connection of the sinh-Gordon equation and differential geometry, in \cite{jwyz} Jost, Wang, Ye and Zhou proved an analogous quantization property as the one in \eqref{quant}, namely that the blow-up limits are multiple of $8\pi$, see Theorem 1.1, Corollary~1.2 and Remark 4.5 in the latter paper. The latter blow-up situation may indeed occur, see \cite{ew} and \cite{gp}. We point out that the assumption $h_1=h_2$ (or $h_1,h_2$ constants) in \cite{jwyz} is crucially used in order to provide a geometric
interpretation of equation \eqref{eq} in terms of constant mean curvature surfaces and harmonic maps (see also \cite{w2}). In this way they transfer the problem into a blow-up phenomenon for harmonic maps. The core of the argument is then to apply a result about no loss of energy during bubbling off for a
sequence of harmonic maps, which was proved in \cite{jo, pa}.
\medskip

The study of the blow-up limits is interesting by itself. However, it yields also important results: we point out here the compactness property of the following sequence of solutions to a variant of \eqref{eq}:
\begin{equation}
\label{variant}
\Delta u_k + \rho_1^k\frac{H_1e^{u_k}}{\int_{\Omega}H_1e^{u_k}}-\rho_2^k\frac{H_2e^{-u_k}}{\int_{\Omega}H_2e^{-u_k}}=0~\mathrm{in}~\Omega\subset\R^2,\quad
u_k=0~\mathrm{on}~\partial\Omega,
\end{equation}
where $\Omega$ is a smooth bounded domain in $\mathbb{R}^2$, $\rho_1^k , \rho_2^k$ are non-negative real parameters and $H_1, H_2$ two \emph{fixed} positive smooth functions (see \cite{bjmr, jev, jev2, jev3, zh} and the references therein). In fact, from the local quantization result in \cite{jwyz} and some standard analysis (see \cite{bat,bm,lwy}) one finds the following global compactness result.
\begin{theorem} \label{compact}
Suppose $\rho_1^k,\rho_2^k$ are two fixed non-negative real numbers and both are not equal to $8\pi\mathbb{N}.$ Then the set of solutions to \eqref{variant} are uniformly bounded.
\end{theorem}
The latter property is a key ingredient in proving both existence and multiplicity results of \eqref{variant}, see for example \cite{bjmr, jev, jev2, jev3}.
\medskip

We return now to the topic of this paper. We shall study here the same subject of \cite{jwyz} in a more general case (i.e., $h_1$, $h_2$ are two different positive $C^3$ functions) by using pure analytic method. The argument is interesting by itself and for the first time it is used for this class of equations.

\

Let $u_k$ be a sequence of blow-up solutions
\begin{equation}
\label{1.2}
\Delta u_k+h_1^ke^{u_k}-h_2^ke^{-u_k}=0,
\end{equation}
with $0$ being its only blow-up point in $B_1$, i.e.:
\begin{equation}
\label{1.3}
\max_{K\subset\subset B_1\setminus\{0\}}|u_k|\leq C(K),\quad \max_{x\in B_1}\{|u_k(x)|\}\rightarrow\infty.
\end{equation}
Throughout the paper we will call $\int_{B_1}h_1^ke^{u_k}$ the energy of $u_k$ (analogously is defined the energy of $-u_k$). We assume moreover
\begin{equation}
\label{1.4}
\frac1C\leq h_i^k(x)\leq C, ~ \|h_i^k(x)\|_{C^3(B_1)}\leq C, \qquad \forall x\in B_1, ~i=1,2
\end{equation}
for some positive constant $C$ and we suppose that $u_k$ has bounded oscillation on $\partial B_1$ and a uniform bound on its energy:
\begin{align}
\label{1.5}
\begin{split}
|u_k(x)-u_k(y)|\leq C,\qquad \forall~x,y\in\partial B_1, \\
\int_{B_1}h_1^ke^{u_k}\leq C, \qquad \int_{B_1}h_2^ke^{-u_k}\leq C,
\end{split}
\end{align}
where $C$ is independent of $k.$
\medskip

Our main result is concerned with the limit energy of $u_k$. Let
\begin{align}
\label{1.6}
\begin{split}
&\sigma_1=\lim_{\delta\rightarrow0}\lim_{k\rightarrow\infty}\frac{1}{2\pi}\int_{B_{\delta}}h_1^ke^{u_k}, \\
&\sigma_2=\lim_{\delta\rightarrow0}\lim_{k\rightarrow\infty}\frac{1}{2\pi}\int_{B_{\delta}}h_2^ke^{-u_k}.
\end{split}
\end{align}
Let $\Sigma$ be the following finite set of points:
\begin{align}
\label{1.8}
\Sigma=\biggr\{	(\sigma_1,\sigma_2)\;=\;\bigr(2m(m+1),2m(m-1)\bigr) \;\mathrm{or}\; \bigr(2m(m-1),2m(m+1)\bigr),\quad m\in\mathbb{N}\biggr\}.
\end{align}

\begin{theorem}
\label{th}
Let $\sigma_i$ and $\Sigma$ be defined as in (\ref{1.6}) and \eqref{1.8}, respectively. Suppose $u_k$ satisfies (\ref{1.2}), (\ref{1.3}), (\ref{1.5}) and $h_i^k$ satisfy (\ref{1.4}). Then $(\sigma_1,\sigma_2)\in\Sigma$.
\end{theorem}

\

\begin{remark}
Theorem \ref{th} yields an improvement of the compactness result in Theorem \ref{compact}, which holds now for arbitrary functions $H_1, H_2$. As a byproduct we get an improvement of both existence and multiplicity results concerning \eqref{variant} in \cite{bjmr, jev2, jev3}. Moreover, it will be crucially used in a forthcoming paper about the Leray-Schauder topological degree associated to \eqref{variant}.
\end{remark}

\begin{remark}
Observe that differently from the Liouville equation \eqref{liouv} and the systems of $n$ equations in \cite{lwz0}, where the blow-up limits (see for example \eqref{quant}) could assume a finite number of possibilities, we obtain here an infinite number of possibilities for \eqref{1.6}, see \eqref{1.8}. The reason for this fact is the different form of the Pohozaev identity associated to the blow-up limits \eqref{1.6}, see Proposition~\ref{pr3.1}.
\end{remark}

\

The first step in the proof of Theorem \ref{th} is to introduce a selection process for describing the situations when blow-up of solutions to (\ref{1.2}) occurs. This argument has been widely used in the framework of prescribed curvature problems, see for example \cite{cl2, ly, s3}. It was later modified by Lin, Wei and Zhang in dealing with general systems of $n$ equations to locate the bubbling area which consists of a finite number of disks, see \cite{lwz0}. Roughly speaking, the idea is that in each disk the blow-up solution have the energy of a globally defined system. We use the same technique for equation (\ref{eq}). Next we prove that in each bubbling disk the energy of at least one of $u_k$ and $-u_k$ is multiple of $4$. Combining then areas closed to each other we deduce that the energy limit of at least one component of $u_k$ and $-u_k$ is multiple of $4$. In this procedure we use the same terminology "group" introduced in \cite{lwz0} to describe bubbling disks closest to each other and relatively far away from other disks. Then, Theorem \ref{th} is a direct consequence of a global Pohozaev identity.
\medskip

The organization of this paper is as follows. In Section \ref{s:selection} we introduce the selection process for the class of equations as in \eqref{eq}, in Section \ref{s:pohozaev} we prove a Pohozaev identity which is the key element in proving Theorem \ref{th}, in Section~\ref{asymptotic} we study the asymptotic behavior of the solutions around the blow-up area and in Section \ref{combination} we finally prove Theorem \ref{th} by a suitable combination of the bubbling areas.

\

\begin{center}
\textbf{Notation}
\end{center}

\

The symbol $B_r(p)$ stands for the open metric ball of
radius $r$ and center $p$. To simplify the notation we will write $B_r$ for balls which are centered at $0$. We will use the notation $a\sim b$ for two comparable quantities $a$ and $b$.

\medskip

Throughout the paper the letter $C$ will stand for large constants which
are allowed to vary among different formulas or even within the same lines.
When we want to stress the dependence of the constants on some
parameter (or parameters), we add subscripts to $C$, as $C_\d$,
etc. We will write $o_{\alpha}(1)$ to denote
quantities that tend to $0$ as $\alpha \to 0$ or $\alpha \to
+\infty$; we will similarly use the symbol
$O_\alpha(1)$ for bounded quantities.

\vspace{1cm}
\section{A selection process for the Sinh-Gordon equation} \label{s:selection}

\medskip

In this section we introduce a selection process for the Sinh-Gordon equation \eqref{eq}. In particular, we will select a finite number of bubbling areas. This will be the first tool to be used in the proof of the main Theorem \ref{th}.
\begin{proposition}
\label{pr2.1}
Let $u_k$ be a sequence of blow-up solutions to (\ref{1.2}) that satisfy (\ref{1.3}) and (\ref{1.5}), and suppose that $h_i^k$ satisfy (\ref{1.4}). Then, there exist finite sequences of points $\Sigma_k:=\{x_1^k,\dots,x_m^k\}$ (all $x_j^k\rightarrow0,~j=1,\dots,m$) and positive numbers $l_1^k,\dots,l_m^k\rightarrow0$ such that
\begin{enumerate}
  \item $\displaystyle{|u_k|(x_j^k)=\max_{B_{l_j^k}(x_j^k)}\{|u_k|\}}$ for $j=1,\dots,m.$ \vspace{0.15cm}
  \item $\exp\bigr(\frac12|u_k|(x_j^k)\bigr)\,l_j^k\rightarrow\infty$ for $j=1,\dots,m.$\vspace{0.1cm}
  \item Let $\e_k=e^{-\frac12M_k},$ where $\displaystyle{M_k=\max_{B_{l_j^k}(x_j^k)}|u_k|}$. In each $B_{l_j^k}(x_j^k)$ we define the dilated functions
	\begin{align} \label{2.1}
  \begin{split}
  &v_1^k(y)=u_k(\e_k y + x_k^j)+2\log\e_k, 	\\
  &v_2^k(y)=-u_k(\e_k y + x_k^j)+2\log\e_k.
  \end{split}
  \end{align}
  Then it holds that one of the $v_1^k,v_2^k$ converges to a function $v$ in $C_{loc}^2(\R^2)$ which satisfies the Liouville equation \eqref{liouv}, while the other one tends to minus infinity over all compact subsets of $\R^2.$
  \item There exits a constant $C_1>0$ independent of $k$ such that
  \begin{equation*}
  |u_k|(x)+2\log\mathrm{dist}(x,\Sigma_k)\leq C_1, \qquad \forall x\in B_1.
  \end{equation*}
\end{enumerate}
\end{proposition}

\begin{proof}
Without loss of generality we may assume that
\begin{equation*}
u_k(x_1^k)=\max_{x\in B_1}|u_k|(x).
\end{equation*}
By assumption we clearly have $x_1^k\rightarrow0.$ Let $(v_1^k,v_2^k)$ be defined as in (\ref{2.1}) with $x_j^k$, $M_k$ replaced by $x_1^k$ and $u_k(x_1^k)$ respectively. Observe that by construction we have $v_i^k\leq0, i=1,2$. Therefore, exploiting the equation \eqref{eq} we can easily see that $|\Delta v_i^k|$ is bounded. By standard elliptic estimate, $|v_i^k(z)-v_i^k(0)|$ is uniformly bounded in any compact subset of $\R^2.$ By construction $v_1^k(0)=0$ and hence $v_1^k$ converges in $C_{loc}^2(\R^2)$ to a function $v_1,$ while the other component is forced to  satisfy $v_2^k\rightarrow-\infty$ over all compact subsets of $\R^2$. The limit of $v_1^k$ satisfies the following equation:
\begin{equation}
\label{2.2}
\Delta v_1+h_1e^{v_1}=0 \qquad \mathrm{in}~\R^2,
\end{equation}
where $h_i=\lim_{k\rightarrow+\infty}h_i^k(x_1^k)$. From (\ref{1.5}), we have
$$\int_{\R^2}h_1e^{v_1}<C.$$
By the classification result due to Chen and Li \cite{cl} it follows that
\begin{align*}
\int_{\R^2}h_1e^{v_1}=8\pi \quad \mathrm{and} \quad v_1(x)=-4\log|x|+O(1),~|x|>2.
\end{align*}
Clearly we can take $R_k\rightarrow\infty$ such that
\begin{equation}
\label{2.3}
v_1^k(y)+2\log|y|\leq C, \qquad |y|\leq R_k.
\end{equation}
In other words we can find $l_1^k\rightarrow0$ such that
\begin{align*}
u_k(x)+2\log|x-x_1^k|\leq C, \qquad |x-x_1^k|\leq l_1^k,
\end{align*}
and
\begin{align*}
e^{\frac12u_1^k(x_1^k)}l_1^k\rightarrow\infty, \qquad \mathrm{as}~k\rightarrow\infty.
\end{align*}

\

\noindent Consider now the function
$$
|u_k(x)|+2\log|x-x_1^k|
$$
and let $q_k$ be the point where $\max_{|x|\leq1} \bigr(|u_k(x)|+2\log|x-x_1^k|\bigr)$ is achieved. Suppose that
\begin{equation}\label{ipo}
\max_{|x|\leq 1}\big(|u_k(x)|+2\log|x-x_1^k|\bigr)\rightarrow\infty.
\end{equation}
Then we define $d_k=\frac12|q_k-x_1^k|$ and
\begin{align*}
\begin{split}
&S_1^k(x)=u_k(x)+2\log\bigr(d_k-|x-q_k|\bigr), \\
&S_2^k(x)=-u_k(x)+2\log\bigr(d_k-|x-q_k|\bigr),
\end{split} \qquad \mbox{in } B_{d_k}(q_k).
\end{align*}
By construction we observe that $S_i^k(x)\rightarrow-\infty$ as $x$ approaches $\partial B_{d_k}(q_k)$ while
\begin{align*}
\max\{S_1^k(q_k),S_2^k(q_k)\}=|u_k(q_k)|+2\log d_k\rightarrow\infty
\end{align*}
by assumption \eqref{ipo}. Let $p_k$ be where $\max_{x\in\overline{B}_{d_k}(q_k)}\{S_1^k,S_2^k\}$ is attained. Without loss of generality, we assume that $S_2^k(p_k)=\max_{x\in\overline{B}_{d_k}(q_k)}\{S_1^k,S_2^k\}$. Then
\begin{equation}
\label{2.4}
-u_k(p_k)+2\log\bigr(d_k-|p_k-q_k|\bigr)\geq \max\{S_1^k(q_k),S_2^k(q_k)\}\rightarrow\infty.
\end{equation}
Let $l_k=\frac12(d_k-|p_k-q_k|)$. By the definition of $p_k$ and $l_k$ we observe that, for $y\in B_{l_k}(p_k)$ it holds
$$
|u_k(y)|+2\log\big(d_k-|y-q_k|\big)\leq -u_k(p_k)+2\log(2l_k),
$$
$$
d_k-|y-q_k|\geq d_k-|p_k-q_k|-|y-p_k|\geq l_k.
$$
Therefore we get
\begin{align}
\label{2.5}
|u_k(y)|\leq -u_k(p_k)+2\log2, \quad \forall~y\in B_{l_k}(p_k).
\end{align}

\

\noindent Now let $R_k=e^{-\frac12u_k(p_k)}l_k$ and define the following functions:
\begin{align*}
&\hat{v}_1^k(y)=u_k(p_k+e^{\frac12u_k(p_k)}y)+u_k(p_k),\\
&\hat{v}_2^k(y)=-u_k(p_k+e^{\frac12u_k(p_k)}y)+u_k(p_k).
\end{align*}
Observe that $R_k\rightarrow\infty$ by (\ref{2.4}). Moreover, $|\Delta\hat{v}_i^k|$ is bounded in $B_{R_k}(0)$. Similarly as before $\hat{v}_2^k(y)$ converges to a function $v_2$ such that
\begin{align*}
\Delta v_2+h_2(p_k)\,e^{v_2}=0.
\end{align*}
On the other hand, $\hat{v}_1^k(y)$ converges uniformly to $-\infty$ over all compact subsets of $\R^2$. Consider now $u_k,-u_k$ in $B_{l_k}(p_k)$ and suppose $x_2^k$ is the point where $\max_{B_{l_k}(p_k)}|u_k|$ is obtained: it is not difficult to see that $-u_k(x_2^k)=\max_{B_{l_k}(p_k)}|u_k|.$ Moreover, we can find $l_2^k$ such that
\begin{align*}
|u_k(x)|+2\log|x-x_2^k|\leq C, \qquad \mathrm{for}~|x-x_2^k|\leq l_2^k.
\end{align*}
By (\ref{2.5}) we have $-u_k(x_2^k) + u_k(p_k)\leq 2\log2$ and we observe that
$$
	\hat{v}_2\Big(e^{-\frac12u_k(p_k)}(x_2^k-p_k)\Big) - \hat{v}_2(0) = -u_k(x_2^k)+u_k(p_k) \leq 2 \log 2.
$$
Therefore we deduce that $e^{-\frac12u_k(p_k)}|x_2^k-p_k|=O(1).$ It follows that we can choose $l_2^k\leq\frac12l_k$ such that $e^{-\frac12u_k(x_2^k)}l_2^k\rightarrow\infty$. Then we re-scale $u_k,-u_k$ around $x_2^k$ and let $v_i^k$ defined in (\ref{2.1}) which will satisfy (1) and (2) in Proposition \ref{pr2.1}. Moreover, it is easy to see that $B_{l_1^k}(x_1^k)\cap B_{l_2^k}(x_2^k)=\emptyset.$
\medskip

\

\noindent In this way we have defined the selection process. To continue it, we let $\Sigma_{k,2}:=\{x_1^k,x_2^k\}$ and consider
\begin{align*}
\max_{x\in B_1}|u_k(x)|+2\log\mathrm{dist}(x,\Sigma_{k,2}).
\end{align*}
If there exists a subsequence such that the quantity above tends to infinity we use the same argument to get $x_3^k$ and $l_3^k$. Since each bubble area contributes a positive energy, the process stops after finite steps due to the bound on the energy (\ref{1.5}). Finally we get
\begin{align*}
\Sigma_k=\{x_1^k,\dots,x_m^k\}
\end{align*}
and it holds
\begin{equation}
\label{2.6}
|u_k(x)|+2\log\mathrm{dist}(x,\Sigma_k)\leq C,
\end{equation}
which concludes the proof.
\end{proof}
\medskip

\begin{lemma}
\label{le2.1}
Let $\Sigma_k=\{x_1^k,\cdots,x_m^k\}$ be the blow-up set obtained in Proposition~\ref{pr2.1}. Then for all $x\in B_1\setminus\Sigma_k$, there exists a constant $C$ independent of $x$ and $k$ such that
\begin{equation*}
|u_k(x_1)-u_k(x_2)|\leq C, \qquad \forall~x_1,x_2\in B\bigr(x,d(x,\Sigma_k)/2\bigr).
\end{equation*}
\end{lemma}

\begin{proof}
Using the Green's representation formula it is not difficult to prove that the oscillation of $u_k$ on $B_1\setminus B_{\frac{1}{10}}$ is finite. Hence we can assume $|x_i|\leq\frac{1}{10}, i=1,2$. Let
\begin{align*}
G(x,\eta)=-\frac{1}{2\pi}\log|x-\eta|+H(x,\eta)
\end{align*}
be the Green's function on $B_1$ with respect to Dirichlet boundary condition. Let $x_0\in B_1\setminus\Sigma_k$ and $x_1,x_2\in B(x_0,d(x_0,\Sigma_k)/2)$. By using the fact $u_k$ has bounded oscillation on $\partial B_1,$ We have
$$
	u_k(x_1) - u_k(x_2) = \int_{B_1} \Big(G(x_1,\eta)-G(x_2,\eta)\Big)
\Big({h}_1^k(\eta)e^{u_k(\eta)}-{h}_2^k(\eta)e^{-u_k(\eta)}\Big)\mathrm{d}\eta + O(1).
$$
Since $|x_i| \leq \frac{1}{10}, i=1,2$ and $\Delta H = 0$ in $B_1$, we can use the bound on the energy \eqref{1.5} to get
$$
\int_{B_1} \Big(H(x_1,\eta)-H(x_2,\eta)\Big) \Big({h}_1^k(\eta)e^{u_k(\eta)}-{h}_2^k(\eta)e^{-u_k(\eta)}\Big)\mathrm{d}\eta = O(1).
$$
Therefore, we are left with proving
$$
	\int_{B_1} \log \frac{|x_1-\eta|}{|x_2-\eta|}\Big({h}_1^k(\eta)e^{u_k(\eta)}-{h}_2^k(\eta)e^{-u_k(\eta)}\Big)\mathrm{d}\eta = O(1).
$$
Let $r_k$ be the distance between $x_0$ and $\Sigma_k$. We distinguish between two cases. Suppose first $\eta\in B_1 \setminus B_{\frac 34 r_k}(x_0)$. Then
$$
\log \frac{|x_1-\eta|}{|x_2-\eta|} = O(1)
$$
and the integration in this region is bounded.

Consider now $\eta \in B_{\frac 34 r_k}(x_0)$ and let
\begin{align*}
&v_1^k(y)=u_k(x_0+r_k y)+2\log r_k,\\
&v_2^k(y)=-u_k(x_0+r_k y)+2\log r_k,
\end{align*}
for $y\in B_{3/4}$. Letting $y_1, y_2$ be the images of $x_1, x_2$ after scaling, namely $x_i = x_0 + r_k y_i, i=1,2$, we have to prove that
$$
	\int_{B_{3/4}} \log \frac{|y_1-\eta|}{|y_2-\eta|}\Big({h}_1^k(x_0 + r_k\eta)e^{v_1^k(\eta)}-{h}_2^k(x_0 + r_k\eta)e^{v_2^k(\eta)}\Big)\mathrm{d}\eta = O(1).
$$
Without loss of generality we may assume that $e_1=(1,0)$ is the image after scaling of the blow-up point in $\Sigma_k$ closest to $x_0$. By Proposition \ref{pr2.1} it holds
\begin{align*}
{v}_i^k(\eta)+2\log|\eta-e_1|\leq C.
\end{align*}
Therefore
\begin{align*}
e^{{v}_i^k(\eta)}\leq C|\eta-e_1|^{-2}.
\end{align*}
Moreover, we notice that $|\eta-e_1|\geq C>0$ for $\eta\in B_{\frac34}.$ Then for $i,j=1,2$, we get
\begin{align*}
\int_{B_{\frac34}}\log|y_j-\eta|{h}_i^k(x_0 + r_k\eta) e^{v_i^k(\eta)}\mathrm{d}\eta\leq
C\int_{B_{\frac34}}\frac{\log|y_j-\eta|}{|\eta-e_1|^2}\mathrm{d}\eta\leq C
\end{align*}
and we are done.
\end{proof}

\vspace{1cm}
\section{Pohozaev identity and related estimates on the energy} \label{s:pohozaev}

\medskip

We establish here a Pohozaev-type identity for the class of equations we are considering. The latter will be a crucial tool in proving the quantization result of Theorem \ref{th}.

We start with some observations and terminology. By Lemma \ref{le2.1} one can see that the behavior of blowup solutions away from the bubbling area can be described just by its spherical average in a neighborhood of a point in $\Sigma_k$. Moreover, the behavior of the solution on a boundary of a ball, say $\partial B_{r}(x_0)$, will play a crucial role in the forthcoming arguments, see for example Remark \ref{rem-fast}. Throughout the paper we will say $u_k$ has fast decay on $\partial B_{r}(x_0)$ if
$$
u_k(x) \leq -2\log|x| - N_k, \qquad \mbox{for } x\in\partial B_{r}(x_0),
$$
for some $N_k\to+\infty$. If instead there exists $C>0$ independent of $k$ such that
$$
u_k(x) \geq -2\log|x| - C, \qquad \mbox{for } x\in\partial B_{r}(x_0),
$$
we say $u_k$ has slow decay on $\partial B_{r}(x_0)$. The same terminology will be used for $-u_k$.

\

For a sequence of bubbling solutions $u_k$ of (\ref{1.2}) recall the definition of local blow-up masses given in \eqref{1.6}. The main result is the following.
\begin{proposition}
\label{pr3.1}
Let $u_k$ satisfy (\ref{1.2}), (\ref{1.3}), (\ref{1.5}) and $h_i^k$ satisfy (\ref{1.4}). Then we have
\begin{align*}
4(\sigma_1+\sigma_2)=(\sigma_1-\sigma_2)^2.
\end{align*}
\end{proposition}
Before we give a proof of Proposition \ref{pr3.1}, we first establish the following auxiliary lemma.
\begin{lemma}
\label{le3.1}
For all $\e_k\rightarrow0$ such that $\Sigma_k\subset B_{\e_k/2}(0)$, there exists $l_k\rightarrow0$ such that $l_k\geq2\e_k$ and
\begin{align}
\label{3.2}
|\overline{u}_k(l_k)|+2\log l_k\rightarrow-\infty,
\end{align}
where $\overline{u}_k(r):=\frac{1}{2\pi r}\int_{\partial B_r}u_k$.
\end{lemma}

\begin{proof}
Given $\e_{k,1}\geq\e_k$ such that $\e_{k,1} \to 0$, there exist $r_{k,1},r_{k,2}\geq\e_{k,1}$ with the following property:
$$
{u}_k(x)+2\log r_{k,1}\rightarrow-\infty, \qquad \forall x\in\partial B_{r_{k,1}},
$$
\begin{equation} \label{3.4}
-{u_k}(x)+2\log r_{k,2}\rightarrow-\infty, \qquad \forall x\in\partial B_{r_{k,2}}.
\end{equation}
Let us focus for example on $u_k$. If the above property is not satisfied, we have some $\e_{k,1}\rightarrow0$ with $\e_{k,1}\geq\e_k$ such that for all $r\geq\e_{k,1}$,
\begin{align*}
\sup_{x\in\partial B_r} \bigr( {u_k}(x)+2\log|x| \bigr) \geq-C,
\end{align*}
for some $C>0.$ But $u_k(x)$ has bounded oscillation on each $\partial B_r$ by Lemma \ref{le2.1}. It follows that
\begin{align*}
u_k(x)+2\log|x|\geq-C
\end{align*}
for some $C$ and all $x\in\partial B_r,~r\geq\e_{k,1}$. This means that
\begin{align*}
e^{u_k(x)}\geq C|x|^{-2}, \qquad \e_{k,1}\leq|x|\leq1.
\end{align*}
Integrating $e^{u_{k}}$ on $B_1\setminus B_{\e_{k,1}}$ we get a contradiction on the uniform energy bound of $\int_{B_1}h_1^ke^{u_k}$ given by \eqref{1.5}. This proves (\ref{3.4}).

\

We start now by taking $\tilde{r}_k\geq\e_k$ so that
\begin{align*}
\overline{u}_k(\tilde{r}_k)+2\log\tilde{r}_k\rightarrow-\infty.
\end{align*}
Suppose $\tilde{r}_k$ is not tending to $0$. Then by Lemma \ref{le2.1} there exists $\hat{r}_k\rightarrow0$ such that
\begin{align} \label{es}
\overline{u}_k(r)+2\log r\rightarrow-\infty, \qquad \mbox{for }\hat{r}_k\leq r\leq\tilde{r}_k.
\end{align}
To prove this we observe that
\begin{align*}
u_k(x)+2\log|x|\leq-N_k, \qquad |x|=\tilde{r}_k
\end{align*}
for some $N_k\rightarrow\infty$. Then, for any fixed $C$, by Lemma \ref{le2.1} we obtain
\begin{align*}
u_k(x)+2\log|x|\leq-N_k+C_0, \qquad \tilde{r}_k/C<|x|<\tilde{r}_k,\\
\end{align*}
Therefore, it is not difficult to prove that $\hat{r}_k$ can be found so that $\frac{\hat{r}_k}{\tilde{r}_k}\rightarrow0$ and (\ref{es}) holds.

Suppose now $\tilde{r}_k\rightarrow0$. Similarly as before we can exploit Lemma \ref{2.1} to find $s_k>\tilde{r}_k$ with $s_k\rightarrow0$ and $\frac{s_k}{\tilde{r}_k}\rightarrow\infty$ such that
\begin{align*}
\overline{u}_k(r)+2\log r\rightarrow-\infty, \qquad \mbox{for } \tilde{r}_k\leq r\leq s_k.
\end{align*}
In both two alternatives we can find $r_k$ with $r_k \in [\hat{r}_k, \tilde{r}_k]$ in the first case, or $r_k \in [\tilde{r}_k, s_k]$ in the second case, such that
\begin{align*}
-\overline{u}_k(r_k)+2\log r_k\rightarrow-\infty.
\end{align*}
In fact, otherwise we would have
\begin{align*}
-\overline{u}_k(r)+2\log r\geq-C, \qquad \mathrm{for}~\hat{r}_k\leq r\leq\tilde{r}_k \quad \mathrm{or} \quad \tilde{r}_k\leq r\leq s_k.
\end{align*}
By the same reason, since by construction $\tilde{r}_k/\hat{r}_k\rightarrow\infty$ or $s_k/\tilde{r}_k\rightarrow\infty$ in each case we get a contradiction to the uniform bound on the energy \eqref{1.5}. Lemma \ref{le3.1} is proved.
\end{proof}
\medskip

\noindent{\em Proof of Proposition \ref{pr3.1}.} We start by observing that there exists $l_k\rightarrow0$ such that $\Sigma_k\subset B_{l_k/2}(0)$, (\ref{3.2}) holds and
\begin{align} \label{st}
\begin{split}
&\frac{1}{2\pi}\int_{B_{l_k}}h_1^ke^{u_k}=\sigma_1+o(1), \\
&\frac{1}{2\pi}\int_{B_{l_k}}h_2^ke^{-u_k}=\sigma_2+o(1).
\end{split}
\end{align}
In fact, one can first choose $l_k$ so that the property \eqref{st} is satisfied and then by Lemma \ref{le3.1} we can further assume that \eqref{3.2} holds true.
Let
\begin{align*}
&v_1^k(y)=u_k(l_ky)+2\log l_k,\\
&v_2^k(y)=-u_k(l_ky)+2\log l_k,
\end{align*}
which satisfy
\begin{align}
\label{3.5}
\begin{split}
\left\{ \begin{array}{l}
\Delta v_1^k(y)+H_1^k(y)\,e^{v_1^k}-H_2^k(y)\,e^{v_2^k}=0, \quad |y|\leq1/l_k,\vspace{0.2cm}\\
\overline{v}_i(1)^k\rightarrow-\infty, \quad i=1,2,
\end{array} \right.
\end{split}
\end{align}
where
$$
H_i^k(y)=h_i^k(l_ky), \qquad i=1,2.
$$
A modification of the Pohozaev-type identity gives us
\begin{align}
\label{3.6}
\begin{split}
&\sum_{i=1}^2\int_{B_{\sqrt{R_k}}}(y\cdot\nabla H_i^k)\,e^{v_i^k}+2\sum_{i=1}^2\int_{B_{\sqrt{R_k}}}H_i^ke^{v_i^k} = \\
=&\sqrt{R_k}\int_{\partial B_{\sqrt{R_k}}}\sum_{i=1}^2H_i^ke^{v_i^k}
+\sqrt{R_k}\int_{\partial B_{\sqrt{R_k}}}\left(|\partial_{\nu}v_1^k|^2- \frac 12 |\nabla v_1^k|^2\right),
\end{split}
\end{align}
where we used $\nabla v_1^k=-\nabla v_2^k$ and $R_k\rightarrow\infty$ will be chosen later. We rewrite the above formula as
\begin{align*}
\mathcal{L}_1+\mathcal{L}_2=\mathcal{R}_1+\mathcal{R}_2+\mathcal{R}_3,
\end{align*}
where the notation is easily understood. First we choose $R_k\rightarrow\infty$ sufficiently smaller than $l_k^{-1}$ so that $\mathcal{L}_1=o(1)$ by $l_k\rightarrow0$. Now we consider $\mathcal{L}_2$. Observe that by Lemma \ref{le2.1}, $v_i^k(y)\rightarrow-\infty$ over all compact subsets of $\mathbb{R}^2\setminus B_{1/2}$. Thus we can choose $R_k$ so that
\begin{align}
\label{3.7}
\int_{B_{R_k}\setminus B_{1}}H_i^ke^{v_i^k}=o(1).
\end{align}
Moreover, by the choice of $l_k$ we have
\begin{align} \label{qua}
\begin{split}
&\frac{1}{2\pi}\int_{B_1}H_1^ke^{v_1^k}=\frac{1}{2\pi}\int_{B_{l_k}}h_1^ke^{u_k}=\sigma_1+o(1), \\
&\frac{1}{2\pi}\int_{B_1}H_2^ke^{v_2^k}=\frac{1}{2\pi}\int_{B_{l_k}}h_2^ke^{-u_k}=\sigma_1+o(1).
\end{split}
\end{align}
Therefore, by (\ref{3.7}) we obtain
\begin{align*}
\mathcal{L}_2=4\pi\sum_{i=1}^2\sigma_i+o(1).
\end{align*}
To estimate $\mathcal{R}_1$ we notice that by (\ref{3.5}) and Lemma \ref{le2.1}
\begin{align}
\label{3.8}
v_i^k(y)+2\log|y|\rightarrow-\infty, \qquad \mathrm{uniformly~in}~1<|y|\leq \sqrt{R_k}.
\end{align}
It follows that $\mathcal{R}_1=o(1).$

\

\noindent Next, we shall estimate the terms $\mathcal{R}_2$ and $\mathcal{R}_3$. To do this we have to estimate $\nabla v_i^k$ on $\partial B_{\sqrt{R_k}}$. Let
\begin{align*}
G_k(y,\eta)=-\frac{1}{2\pi}\log|y-\eta|+H_k(y,\eta)
\end{align*}
be the Green's function on $B_{l_k^{-1}}$ with respect to Dirichlet boundary condition. The regular part is expressed as
\begin{align*}
H_k(y,\eta)=\frac{1}{2\pi}\log\frac{|y|}{l_k^{-1}}\left|\frac{l_k^{-2}y}{|y|^2}-\eta\right|
\end{align*}
and it holds
\begin{align}
\label{3.9}
\nabla_yH_k(y,\eta)=O(l_k), \qquad \mbox{for } y\in\partial B_{\sqrt{R_k}},~\eta\in B_{l_k^{-1}}.
\end{align}
We start by estimating $\nabla v_1^k$ on $\partial B_{\sqrt{R_k}}$. By the Green's representation formula,
\begin{align*}
v_1^k(y)=\int_{B_{l_k^{-1}}}G(y,\eta)\Big(H_1^ke^{v_1^k}-H_2^ke^{v_2^k}\Big)+F_k,
\end{align*}
where $F_k$, which is the boundary term, is a harmonic function satisfying $F_k=v_i^k$ on $\partial B_{l_k^{-1}}$. In particular $F_k$ has bounded oscillation on $\partial B_{l_k^{-1}}$. It follows that $F_k-C_k=O(1)$ for some $C_k,$ which yields $|\nabla F_k(y)|=O(l_k).$
\begin{align}
\begin{split}
\label{3.10}
\nabla v_1^k(y)=&\int_{B_{l_k^{-1}}}\nabla_yG(y,\eta)\Big(H_1^ke^{v_1^k}-H_2^ke^{v_2^k}\Big)\mathrm{d}\eta+\nabla F_k(y)\\
=&-\frac{1}{2\pi}\int_{B_{l_k^{-1}}}\frac{y-\eta}{|y-\eta|^2}\Big(H_1^ke^{v_1^k}-H_2^ke^{v_2^k}\Big)\mathrm{d}\eta+O(l_k).
\end{split}
\end{align}
In order to estimate the integral of (\ref{3.10}) we divide the domain into few regions. We first observe that
$$
\frac{1}{|y-\eta|}\sim\frac{1}{|\eta|}\leq o\left(R_k^{-\frac12}\right), \qquad \mbox{for } y\in\partial B_{\sqrt{R_k}}, \quad \eta\in B_{l_k^{-1}}\setminus B_{R_k^{2/3}} .
$$
Hence, using the bound of the energy \eqref{1.5}, the integral over $B_{l_k^{-1}}\setminus B_{R_k^{2/3}}$ is $o(1)R_k^{-\frac12}$.

Consider now the integral over $B_1$: we have
\begin{align*}
\frac{y-\eta}{|y-\eta|^2}=\frac{y}{|y|^2}+O\left(\frac{1}{|y|^2}\right), \qquad \mbox{for } y\in\partial B_{\sqrt{R_k}}, \quad \eta\in B_1,
\end{align*}
which, recalling \eqref{qua}, yields
\begin{align*}
-\frac{1}{2\pi}\int_{B_1}\frac{y-\eta}{|y-\eta|^2}\Big(H_1^ke^{v_1^k}-H_2^ke^{v_2^k}\Big)
=\left(-\frac{y}{|y|^2}+O\left(\frac{1}{|y|^2}\right)\right)\Big(\sigma_1-\sigma_2+o(1)\Big).
\end{align*}
As we will see this will be the leading term.

For the integral over the region $B_{\sqrt{R_k}/2}\setminus B_1$ we observe that
$$
\frac{1}{|y-\eta|}\sim\frac{1}{|y|}, \qquad \mbox{for } y\in\partial B_{\sqrt{R_k}}, \quad \eta\in B_{\sqrt{R_k}/2}\setminus B_1.
$$
By the latter estimate and by (\ref{3.7}) we get
\begin{align*}
\int_{B_{\sqrt{R_k}/2}\setminus B_1}\frac{y-\eta}{|y-\eta|^2}\Big(H_1^ke^{v_1^k}-H_2^ke^{v_2^k}\Big)=o(1)|y|^{-1}.
\end{align*}
Similarly one gets
\begin{align*}
\int_{B_{R_k^{2/3}}\setminus \bigr(B_1 \cup B_{\frac{|y|}{2}}(y)\bigr)}\frac{y-\eta}{|y-\eta|^2}\Big(H_1^ke^{v_1^k}-H_2^ke^{v_2^k}\Big)=o(1)|y|^{-1}.
\end{align*}
Moreover, for the integral in $B_{\frac{|y|}{2}}(y)$ we use $e^{v_i^k(\eta)}=o(1)|\eta|^{-2}$ to get
\begin{align*}
\int_{B_{\frac{|y|}{2}}(y)}\frac{y-\eta}{|y-\eta|^2}\Big(H_1^ke^{v_1^k}-H_2^ke^{v_2^k}\Big)=o(1)|y|^{-1}.
\end{align*}
Finally, combing all the estimates we deduce
\begin{align*}
\nabla v_1^k(y)=\left(-\frac{y}{|y|^2}\right)\biggr(\sigma_1-\sigma_2+o(1)\biggr)+o\left(|y|^{-1}\right), \qquad \mbox{for } y\in\partial B_{\sqrt{R_k}}.
\end{align*}
Exploiting the latter formula in $\mathcal{R}_2$ and $\mathcal{R}_3$ we get
$$
	\mathcal{R}_2 + \mathcal{R}_3 = \pi(\sigma_1-\sigma_2)^2+o(1).
$$
Therefore, we end up with
\begin{align*}
4(\sigma_1+\sigma_2)=(\sigma_1-\sigma_2)^2+o(1).
\end{align*}
Hence, Proposition \ref{pr3.1} is proved.
\begin{flushright}
$\square$
\end{flushright}

\

\begin{remark} \label{rem-fast}
By the proof of Proposition \ref{pr3.1} one observes the following fact: the fast decay property is crucial in evaluating the Pohozaev identity, more precisely the term $\mathcal{R}_1$. Moreover, letting $\Sigma_k'\subseteq\Sigma_k$ suppose that
$$\mathrm{dist}\bigr(\Sigma_k',\partial B_{l_k}(p_k)\bigr)=o(1)\,\mathrm{dist}\bigr(\Sigma_k\setminus\Sigma_k',\partial B_{l_k}(p_k)\bigr).$$
Suppose moreover both components $u_k, -u_k$ have fast decay on $\partial B_{l_k}(p_k)$, namely
$$
	|u_k(x)| \leq -2\log|x| -N_k, \qquad  \mbox{for } x\in \partial B_{l_k}(p_k),
$$
for some $N_k\to+\infty$. Then, we can evaluate a local Pohozaev identity and get
\begin{align*}
\bigr(\tilde{\sigma}_1^k(l_k)-\tilde{\sigma}_2^k(l_k)\bigr)^2=4\bigr(\tilde{\sigma}_1^k(l_k)+\tilde{\sigma}_2^k(l_k)\bigr)+o(1),
\end{align*}
where
\begin{align*}
&\tilde{\sigma}_1^k(l_k) = \frac{1}{2\pi}\int_{B_{l_k}(p_k)}h_1^k e^{u_k} \\
&\tilde{\sigma}_2^k(l_k) = \frac{1}{2\pi}\int_{B_{l_k}(p_k)}h_2^k e^{-u_k}.
\end{align*}
Observe that if $B_{l_k}(p_k) \cap \Sigma_k=\emptyset$ then $\tilde{\sigma}_i^k(l_k)=o(1), i=1,2$ and the above formula clearly holds.

This fact will be used in the forthcoming arguments.
\end{remark}

\vspace{1cm}
\section{Asymptotic behavior of solutions around each blow-up point} \label{asymptotic}

\medskip

The goal in this section is to get some energy classification in each blow-up area. We will see in the sequel how the fast decaying property of the solutions plays a crucial role in determining the local energy. Once we obtain the classification around each blow-up point, in Section \ref{combination} we combine them together.

By considering suitable translated functions we may assume without loss of generality that $0\in \Sigma_k$ for any $k$. Let $\tau_k=\frac12\mathrm{dist}(0,\Sigma_k\setminus\{0\})$ we consider the energy limits of $h_1^ke^{u_k}$ and $h_2^ke^{-u_k}$ in $B_{\tau_k}$. Define
\begin{align}
\label{4.2}
\begin{split}
&v_1^k=u_k(\delta_ky)+2\log\delta_k, \\
&v_2^k=-u_k(\delta_ky)+2\log\delta_k,
\end{split} \qquad |y|\leq\tau_k/\delta_k,
\end{align}
where $-2\log\delta_k=\max_{x\in B(0,\tau_k)}|u_k|.$ Thus the equation for $v_i^k$ is
\begin{align*}
\Delta v_1^k(y)+h_1^k(\delta_ky)\,e^{v_1^k(y)}-h_2^k(\delta_ky)\,e^{v_2^k(y)}=0, \qquad |y|\leq\tau_k/\delta_k.
\end{align*}
By the definition of the selection process we have $\tau_k/\delta_k\rightarrow\infty$, see Proposition \ref{pr2.1}.

Observe moreover that
\begin{align*}
&\int_{B_{\tau_k}(0)}h_1^k(x)\, e^{u_k(x)}\mathrm{d}x=\int_{B_{\tau_k/\delta_k}(0)}h_1^k(\d_k y)\, e^{v_1^k(y)}\mathrm{d}y, \\
&\int_{B_{\tau_k}(0)}h_2^k(x)\, e^{-u_k(x)}\mathrm{d}x=\int_{B_{\tau_k/\delta_k}(0)}h_2^k(\d_k y)\, e^{v_2^k(y)}\mathrm{d}y.
\end{align*}
Therefore
\begin{align}
\label{4.4}
\int_{B_{\tau_k}(0)}h_1^k(x)e^{u_k(x)}\mathrm{d}x=O(1)e^{\overline{v}_1^k(\partial B_{\tau_k/\delta_k}(0))},~
\int_{B_{\tau_k}(0)}h_1^k(x)e^{-u_k(x)}\mathrm{d}x=O(1)e^{\overline{v}_2^k(\partial B_{\tau_k/\delta_k}(0))},
\end{align}
Define the following local masses:
\begin{align}
\label{local}
\begin{split}
&{\sigma}_1^k(r)=\frac{1}{2\pi}\int_{B_{r}}h_1^ke^{u_k}, \\
&{\sigma}_2^k(r)=\frac{1}{2\pi}\int_{B_{r}}h_2^ke^{-u_k}.
\end{split}
\end{align}
The main result of this section is the following.
\begin{proposition}
\label{pr4.1}
Suppose (\ref{1.2})-(\ref{1.5}) hold for $u_k,$ $h_i^k$ and recall the definition in \eqref{local}. For any $s_k\in(0,\tau_k)$ such that both $u_k, -u_k$ have fast decay on $\partial B_{s_k}$, i.e.
\begin{align} \label{fast}
|u_k(x)|\leq-2\log|x|-N_k, \qquad \mbox{for } |x|=s_k \mbox{ and some } N_k\rightarrow\infty,
\end{align}
we have that $(\sigma_1^k(s_k),\sigma_2^k(s_k))$ is a $o(1)$ perturbation of one of the two following types:
\begin{align*}
\bigr(2m(m+1),2m(m-1)\bigr) \quad \mathrm{or} \quad \bigr(2m(m-1),2m(m+1)\bigr),
\end{align*}
for some $m\in \mathbb{N}$. In particular, they are both multiple of $4 + o(1)$.

On $\partial B_{\tau_k}$, either both $u_k, -u_k$ have fast decay as in \eqref{fast} and the conclusion is as before, or one component has fast decay while the other one is not fast decay. Suppose for example $-u_k$ has not the fast decay property, i.e.
\begin{align*}
-u_k(x)+2\log|x|\geq-C, \qquad \mbox{for } |x|=\tau_k \mbox{ and some } C>0,
\end{align*}
while for $u_k$ it holds
\begin{align*}
u_k(x) \leq-2\log|x|-N_k, \qquad \mbox{for } |x|=s_k \mbox{ and some } N_k\rightarrow\infty.
\end{align*}
Then $\sigma_1^k(\tau_k)\in 4\pi\mathbb{N} + o(1)$.

In particular, in any case at least one of the two components $u_k, -u_k$ has the local energy in $B_{\tau_k}$ multiple of $4 + o(1)$.
\end{proposition}

\begin{proof}
Let $v_i^k$ be defined in (\ref{4.2}). Observe that by construction one of the $v_i^k$'s converges while the other one goes to minus infinity over all compact subsets of $\mathbb{R}^2$ (see the argument in Proposition \ref{pr2.1}), namely we have just a partially blown-up situation. Without loss of generality we assume that $v_1^k$ converges to $v_1$ in $C_{loc}^2(\R^2)$ and $v_2^k$ tends to minus infinity over any compact subset of $\R^2.$ The equation for $v_1$ is
\begin{align*}
\Delta v_1+h_1e^{v_1}=0~\mathrm{in}~\R^2,\qquad\int_{\R^2}h_1e^{v_1}<\infty,
\end{align*}
where $h_1=\lim_{k\rightarrow\infty}h_1^k(0)$. By the classification result of Chen-Li \cite{cl}, we have
\begin{align*}
\int_{\R^2}h_1e^{v_1}=8\pi
\end{align*}
and
\begin{align*}
v_1(y)=-4\log|y|+O(1),\qquad|y|>1.
\end{align*}
Therefore, we can take $R_k\rightarrow\infty$ (we assume $R_k=o(1)\tau_k/\delta_k$) such that
\begin{align} \label{m=4}
\frac{1}{2\pi}\int_{B_{R_k}}h_1^k(\delta_ky)\,e^{v_1^k}=4+o(1),
\end{align}
and
\begin{align} \label{m=0}
\frac{1}{2\pi}\int_{B_{R_k}}h_2^k(\delta_ky)\,e^{v_2^k}=o(1).
\end{align}
For $r\geq R_k$ we clearly have
\begin{align*}
\sigma_i^k(\delta_k r)=\frac{1}{2\pi}\int_{B_r}h_i^k(\delta_ky)\,e^{v_i^k}.
\end{align*}
Up to now we get by \eqref{m=4} and \eqref{m=0} that
$$
\sigma_1^k(\delta_k R_k)=4+o(1), \qquad  \sigma_2^k(\delta_k R_k) = o(1).
$$

\

\noindent Let $\overline{v}_i^k(r)$ be the average of ${v}_i^k$ on $\partial B_r$, $i=1,2$. It will be important to study $\frac{d}{dr}\overline{v}_i^k(r), i=1,2$. In fact if
$$
\frac{d}{dr}\left(\overline{v}_i^k(r) +2\log r \right) >0, \qquad \mbox{for some } i,
$$
there is a possibility that for some larger radius $s$, $v_i^k$ becomes a slow decay component on $\partial B_s$.

The key fact is to observe that
\begin{align} \label{deriv}
\begin{split}
&\frac{d}{dr}\overline{v}_1^k(r)=\frac{-\sigma_1^k(\delta_k r)+\sigma_2^k(\delta_k r)}{r}, \\
&\frac{d}{dr}\overline{v}_2^k(r)=\frac{\sigma_1^k(\delta_k r)-\sigma_2^k(\delta_k r)}{r},
\end{split} \qquad R_k\leq r\leq\tau_k/\delta_k.
\end{align}
Clearly we have
\begin{align*}
R_k\frac{d}{dr}\overline{v}_1^k(R_k)=-4+o(1), \qquad R_k\frac{d}{dr}\overline{v}_2^k(R_k)=4+o(1).
\end{align*}

\

\noindent To continue the proof of Proposition \ref{pr4.1} we prove now two auxiliary lemmas.
\begin{lemma}
\label{le4.1}
Suppose there exists $L_k\in(R_k,\tau_k/\delta_k)$ such that
\begin{align}
\label{4.6}
v_i^k(y)\leq-2\log|y|-N_k, \qquad \mbox{for } R_k\leq|y|\leq L_k,~i=1,2
\end{align}
for some $N_k\rightarrow\infty$. Then $\sigma_i^k$ does not change much from $\delta_kR_k$ to $\delta_kL_k$: more precisely we have
\begin{align*}
\sigma_i^k(\delta_kL_k)=\sigma_i^k(\delta_kR_k)+o(1), \qquad i=1,2.
\end{align*}
\end{lemma}

\begin{proof}
Suppose the statement is false: then there exists $i$ such that $\sigma_i^k(\delta_kL_k)>\sigma_i^k(\delta_kR_k)+\delta$ for some $\delta>0$. Let us choose $\tilde{L}_k\in(R_k,L_k)$ such that
\begin{align}
\label{4.7}
\max_{i=1,2}\Big(\sigma_i^k(\delta_k\tilde{L}_k)-\sigma_i^k(\delta_kR_k)\Big)=\e,
\end{align}
where $\e>0$ is taken sufficiently small. Then by using \eqref{deriv} we have
\begin{align}
\label{4.8}
\frac{d}{dr}\overline{v}_1^k(r)\leq\frac{-4+\e+o(1)}{r}\leq-\frac{2+\e}{r}, \qquad R_k \leq r \leq \tilde L_k.
\end{align}
By Lemma \ref{le2.1} we observe that
$$
	v_i^k(x) = \overline{v}_i^k(|x|) + O(1), \qquad x \in B_{\tau_k/\d_k},
$$
where $\overline{v}_i^k(|x|)$ is the average of $v_i$ on $\partial B_{|x|}$. Reasoning as above and exploiting \eqref{4.8} jointly with \eqref{4.6} it is not difficult to show that
\begin{align*}
\int_{B_{\tilde L_k}\setminus B_{R_k}}e^{v_1^k}=O(1)\int_{B_{\tilde L_k}\setminus B_{R_k}}e^{\overline{v}_1^k(\tilde L_k)}=o(1).
\end{align*}
In other words we have $\sigma_1^k(\delta_k\tilde{L}_k)=\sigma_1^k(\delta_kR_k)+o(1).$

It follows that the maximum in \eqref{4.7} is attained for $i=2$, i.e.
\begin{align}\label{m=e}
\sigma_2^k(\delta_k\tilde{L}_k)=\sigma_2^k(\delta_kR_k)+\e.
\end{align}
On the other hand, since (\ref{4.6}) holds, as observed in Remark \ref{rem-fast} we get
\begin{align*}
\bigr(\sigma_1^k(\delta_k\tilde{L}_k)-\sigma_2^k(\delta_k\tilde{L}_k)\bigr)^2=
4\bigr(\sigma_1^k(\delta_k\tilde{L}_k)+\sigma_2^k(\delta_k\tilde{L}_k)\bigr) + o(1).
\end{align*}
Now we observe that
$$
\sigma_1^k(\delta_k\tilde{L}_k)=\sigma_1^k(\delta_kR_k)+o(1) = 4 + o(1),
$$
where we used \eqref{m=4}. Therefore we deduce that
$$
\sigma_2^k(\delta_k\tilde{L}_k)=o(1) \qquad \mbox{or} \qquad \sigma_2^k(\delta_k\tilde{L}_k)=12+o(1),
$$
which contradicts
$$
\sigma_2^k(\delta_k\tilde{L}_k)=\sigma_2^k(\delta_kR_k)+\e = \e+o(1),
$$
see \eqref{m=0} and \eqref{m=e}. Thus Lemma \ref{le4.1} is established.
\end{proof}

\

\noindent By the argument in Lemma \ref{le4.1} we observe the following fact: for $r\geq R_k$ either both $v_1, v_2$ have fast decay up to $\partial B_{\tau_k/\delta_k}$, namely
\begin{align}
\label{4.9}
v_i^k(y)\leq-2\log|y|-N_k, \qquad R_k\leq|y|\leq\tau_k/\delta_k,~i=1,2,
\end{align}
for some $N_k\to+\infty$, or there exists $L_k\in(R_k,\tau_k/\delta_k)$ such that $v_2$ has the following slow decay
\begin{align}
\label{4.10}
v_2^k(y)\geq-2\log L_k-C, \qquad |y|=L_k,
\end{align}
for some $C>0,$ while
\begin{align}
\label{4.11}
v_1^k(y)\leq-2\log|y| - N_k, \qquad R_k\leq|y|\leq L_k,
\end{align}
for some $N_k\to+\infty$. Indeed, we have noticed in Lemma \ref{le4.1} that if the local energy changes, $\sigma_2^k$ has to change first. Moreover, we have seen that $L_k$ can be chosen so that $\sigma_2^k(\delta_kL_k)-\sigma_2^k(\delta_kR_k)=\e$ for some $\e>0$ small. The following lemma treats the latter situation.

\begin{lemma}
\label{le4.2}
Suppose there exists $L_k\geq R_k$ such that (\ref{4.10}) and (\ref{4.11}) hold. We assume moreover that $L_k=o(1)\tau_k/\delta_k$. Then there exists $\tilde{L}_k$ such that $\tilde{L}_k/L_k\rightarrow\infty$ and $\tilde{L}_k=o(1)\tau_k/\delta_k$ with the following property:
\begin{align}
\label{4.12}
v_i^k(y)\leq-2\log|y|-N_k, \qquad |y|=\tilde{L}_k,~i=1,2,
\end{align}
for some $N_k\rightarrow\infty.$ Moreover
\begin{align}
\label{4.14}
\sigma_1^k(\delta_k\tilde{L}_k)=4+o(1), \qquad \sigma_2^k(\delta_k\tilde{L}_k)=12+o(1).
\end{align}
\end{lemma}

\begin{proof}
First we observe that by the choice of $L_k$ and being $\sigma_2^k(\delta_k R_k)=o(1)$ we can assume that $\sigma_2^k(\delta_kL_k)\leq\e$ for some $\e>0$ small. By the same reason in Lemma \ref{le4.1} we have
\begin{align*}
\frac{d}{dr}\overline{v}_1^k(r)\leq\frac{-4+\e +o(1)}{r}, \qquad R_k\leq r\leq L_k.
\end{align*}
Now we claim there exists $N>1$ such that
\begin{align}
\label{4.15}
\sigma_2^k(\delta_k(NL_k))\geq 8+o(1).
\end{align}
Suppose this does not hold. Then there exist $\e_0>0$ and $\tilde{R}_k\rightarrow\infty$ such that
\begin{align}
\label{4.16}
\sigma_2^k(\delta_k\tilde{R}_kL_k)\leq 8-\e_0.
\end{align}
Moreover, $\tilde{R}_k$ can be chosen to tend to infinity slowly so that by Lemma \ref{le2.1} and \eqref{4.11} we get
\begin{align}
\label{4.17}
v_1^k(y)\leq-2\log|y| - N_k, \qquad L_k\leq|y|\leq\tilde{R}_kL_k.
\end{align}
By Lemma \ref{le4.1}, (\ref{4.17}) implies $\sigma_1^k(\delta_kL_k)=\sigma_1^k(\delta_k\tilde{R}_kL_k)+o(1)$. Thus by (\ref{4.16})
\begin{align}
\label{4.18}
\frac{d}{dr}\overline{v}_2^k(r)\geq\frac{-4+\e+o(1)}{r}.
\end{align}
From (\ref{4.18}) and \eqref{4.10} it is not difficult to show
$$
\int_{B_{L_k\tilde{R}_k}\setminus B_{L_k}}e^{v_2^k}\rightarrow\infty,
$$
which contradicts the energy bound in (\ref{1.5}). Therefore (\ref{4.15}) holds.

\

\noindent By Lemma \ref{le2.1} we have
\begin{align*}
v_i^k(y)+2\log NL_k=\overline{v}_i^k(NL_k)+2\log (NL_k)+O(1), \qquad i=1,2,~|y|=NL_k.
\end{align*}
Therefore, by the assumptions we get
\begin{align*}
&\overline{v}_1^k(NL_k)\leq-2\log(NL_k)-N_k,\\
&\overline{v}_2^k(NL_k)\geq-2\log(NL_k)-C\geq-2\log(L_k)-C.
\end{align*}
Furthermore we can assert that
\begin{align*}
\overline{v}_2^k((N+1)L_k)\geq-2\log L_k-C,
\end{align*}
which, jointly with \eqref{4.15}, yields
\begin{align*}
\int_{B_{(N+1)L_k}}h_2^k(\delta_ky)\,e^{v_2^k(y)}dy\geq 8+\e_0,
\end{align*}
for some $\e_0>0$.

By the latter estimate we get
\begin{align*}
\frac{d}{dr}\overline{v}_2^k(r)\leq-\frac{2+\e_0}{r}, \qquad \mbox{for } r=(N+1)L_k.
\end{align*}
Therefore we can take $\tilde{R}_k\rightarrow\infty$ slowly such that $\tilde{R}_kL_k=o(1)\tau_k/\delta_k$ and
\begin{align*}
v_2^k(y)\leq(-2-\e_0)\log|y|-N_k, \qquad|y|=\tilde{R}_kL_k,\\
v_1^k(y)\leq-2)\log|y|-N_k, \qquad L_k\leq|y|\leq\tilde{R}_kL_k,
\end{align*}
where we have used also Lemma \ref{le2.1}. By Lemma \ref{le4.1} and \eqref{m=4} we have
\begin{align*}
\sigma_1^k(\delta_k\tilde{R}_kL_k)=\sigma_1^k(\delta_kL_k)+o(1)=4+o(1).
\end{align*}
Moreover, on $\partial B_{\tilde{R}_kL_k}$ both components have fast decay. Thus as in Remark \ref{rem-fast} we can compute the Pohozaev identity and observing \eqref{4.15} holds we get
\begin{align*}
\sigma_2^k(\delta_k\tilde{R}_kL_k)=12+o(1).
\end{align*}
Letting $\tilde{L}_k=\tilde{R}_kL_k$ we conclude the proof.
\end{proof}

\

\noindent Returning to the proof of Proposition \ref{pr4.1}, we are left with the region $\tilde{L}_k\leq|y|\leq\tau_k/\delta_k$. We distinguish between two cases.

\

Suppose first $L_k=O(1)\tau_k/\delta_k.$ Then by Lemma \ref{le2.1} we directly conclude that one component has fast decay while the other one has slow decay, see for example \eqref{4.10} and \eqref{4.11}. Moreover, we have observed that the energy in $B_{\tau_k}$ of the fast decaying component is a small perturbation of $4$. This is exactly the second alternative of Proposition \ref{pr4.1} and therefore the proof is concluded.

\

Suppose now $L_k=o(1)\tau_k/\delta_k$. In this case $\tilde{L}_k$ can chosen so that $o(1)\tau_k/\delta_k$. By using the local energies given by Lemma \ref{le4.2} we have
\begin{align*}
\frac{d}{dr}\overline{v}_1^k(r)=\frac{8+o(1)}{r}, \quad
\frac{d}{dr}\overline{v}_2^k(r)=-\frac{8+o(1)}{r}, \qquad \mbox{for } r=\tilde{L}_k.
\end{align*}
It follows that
\begin{align*}
\frac{d}{dr}\overline{v}_2^k(r)\leq-\frac{2+\e}{r}, \qquad r=\tilde{L}_k,
\end{align*}
for some $\e>0.$ As in Lemma \ref{le4.1} we conclude that $\sigma_2^k(r)$ does not change for $r\geq\tilde{L}_k$ unless $\sigma_1^k$ changes. By the same ideas of Lemmas \ref{le4.1}, \ref{le4.2} and by the argument of the first case $L_k=O(1)\tau_k/\delta_k$, either $v_1^k$ has slow decay up to $B_{\tau_k/\delta_k}$ or there is $\hat{L}_k=o(1)\tau_k/\delta_k$ such that
\begin{align*}
\sigma_1^k(\delta\hat{L}_k)=24+o(1), \qquad \sigma_2^k(\delta\hat{L}_k)=12+o(1).
\end{align*}
By the latter local energies we deduce
\begin{align*}
\frac{d}{dr}\overline{v}_1^k(r)=-\frac{12+o(1)}{r}, \quad
\frac{d}{dr}\overline{v}_2^k(r)=\frac{12+o(1)}{r}, \qquad \mbox{for } r=\hat{L}_k.
\end{align*}
Thus as before
\begin{align*}
\frac{d}{dr}\overline{v}_1^k(r)\leq-\frac{2+\e}{r},~r=\hat{L}_k,
\end{align*}
for some $\e>0.$ Now $\sigma_1^k(r)$ does not change for $r\geq\hat{L}_k$ unless $\sigma_2^k$ changes. By repeating the argument we get either $v_2^k$ has slow decay up to $B_{\tau_k/\delta_k}$ or there is $\overline{L}_k=o(1)\tau_k/\delta_k$ such that
\begin{align*}
\sigma_1^k(\delta\overline{L}_k)=24+o(1), \qquad \sigma_2^k(\delta\overline{L}_k)=40+o(1).
\end{align*}
Since after each step one of the local masses changes by a positive number, using  the uniform bound on the energy (\ref{1.5}) the process stops after finite steps. Eventually we get Proposition \ref{pr4.1}.
\end{proof}

\vspace{1cm}
\section{Combination of bubbling areas and proof of Theorem \ref{th}} \label{combination}

\medskip

In this section we present an argument for combining the blow-up areas. This strategy has been already used in several frameworks, see the Introduction for more details. The idea is the following: we start by considering blow-up points which are close to each other and we get a quantization property for each group, see the definition of group below. In particular, in each group the local energy of at least one component is a small perturbation of $4n$, for some $n\in\mathbb{N}$. Similarly, we combine the groups and we get the total energy of at least one component is a small perturbation of $4n$, for some $n\in\mathbb{N}$. Then, the conclusion follows by applying a global Pohozaev identity.

\medskip

\noindent {\bf Definition.} Let $Q_k=\{p_1^k,\cdots,p_q^k\}$ be a subset of $\Sigma_k$ with more than one point in it. $Q_k$ is called a group if
\begin{enumerate}
  \item dist$(p_i^k,p_j^k)\sim $ dist$(p_s^k,p_t^k)$,\\
 where $p_i^k,p_j^k,p_s^k,p_t^k$ are any points in $Q_k$ such that $p_i^k\neq p_j^k$ and $p_t^k\neq p_s^k.$ \vspace{0.3cm}

  \item $\dfrac{\mbox{dist}(p_i^k,p_j^k)}{\mbox{dist}(p_i^k,p_k)}\rightarrow0$, \\
  for any $p_k\in\Sigma_k\setminus Q_k$ and for all $p_i^k,p_j^k\in Q_k$ with $p_i^k\neq p_j^k.$
\end{enumerate}

\

\noindent {\em Proof of Theorem \ref{th}:} As in Section \ref{asymptotic}, by considering suitable translated functions we may assume without loss of generality that $0\in \Sigma_k$ for any $k$. Let $2\tau_k$ be the distance between $0$ and $\Sigma_k\setminus\{0\}.$ To describe the group $G_0$ that contains $0$ we proceed in the following way: if for any $z_k\in\Sigma_k\cap\partial B(0,2\tau_k)$ we have $dist(z_k,\Sigma_k\setminus\{z_k\})\sim\tau_k,$ then $G_0$ contains at least two points. On the other hand, if there exists $z_k'\in\partial B(0,2\tau_k)\cap\Sigma_k$ such that $\tau_k/dist(z_k',\Sigma_k\setminus z_k')\rightarrow\infty$ we let $G_0$ be $0$ itself. By the definition of group, all points of $G_0$ are in $B(0,N\tau_k)$ for some $N$ independent of $k$. Let
\begin{align*}
\begin{split}
&\tilde{v}_1^k(y)=u_k(\tau_ky)+2\log\tau_k,  \\
&\tilde{v}_2^k(y)=-u_k(\tau_ky)+2\log\tau_k,
\end{split}
\qquad |y|\leq\tau_k^{-1},
\end{align*}
which satisfy
\begin{equation}
\label{6.1}
\Delta\tilde{v}_1^k(y)+h_1^k(\tau_ky)\,e^{\tilde{v}_1^k(y)}-h_2^k(\tau_ky)\,e^{\tilde{v}_2^k(y)}=0, \qquad |y|\leq\tau_k^{-1}.
\end{equation}

Let $0,Q_1,\cdots,Q_s$ be the images of members of $G_0$ after scaling from $y$ to $\tau_ky.$ We observe that $Q_i\in B_N.$ By Proposition \ref{pr4.1} at least one of $\tilde{v}_i^k$ decays fast on $\partial B_1.$ Without loss of generality we assume
\begin{equation*}
\tilde{v}_1^k\leq -2\log|y|-N_k, \qquad \mathrm{on}~\partial B_1,
\end{equation*}
for some $N_k\rightarrow\infty$. Moreover, we know by Proposition \ref{pr4.1} that
\begin{equation*}
\sigma_1^k(\tau_k)=4\tilde n + o(1),
\end{equation*}
for some $\tilde n\in\mathbb{N}$.

Furthermore, by Lemma \ref{le2.1} we can assert
\begin{equation*}
\tilde{v}_1^k\leq -2\log|y|-N_k, \qquad \mathrm{on}~\partial B_1(Q_j), \quad j=1,\dots,s.
\end{equation*}
Still by using Proposition \ref{pr4.1} we conclude that the the local energies around $Q_j$ are
$$
	\frac{1}{2\pi} \int_{B_1(Q_j)} h_1(\tau_k y)\, e^{\tilde{v}_1^k} = 4n_j + o(1), \qquad j=1,\dots,s,
$$
for some $n_j\in\mathbb{N}, j=1,\dots,s$.
Let $2\tau_kL_k$ be the distance from $0$ to the nearest group from $G_0.$ By the definition of group we have $L_k\rightarrow\infty.$ By using Lemma \ref{le2.1} and Lemma \ref{le4.1}, using the same reason in Lemma \ref{le3.1} we can find $\tilde{L}_k\leq L_k,~\tilde{L}_k\rightarrow\infty$ slowly such that the energy of $\tilde{v}_1^k$ in $B_{\tilde{L}_k}(0)$ does not change so much and such that $\tilde{v}_2^k$ has fast decay on $\partial B_{\tilde{L}_k}(0)$:
\begin{equation}
\label{6.2}
\sigma_1^k(\tau_k \tilde L_k)=4n + o(1),
\end{equation}
for some $n\in\mathbb{N}$ and
\begin{equation}
\label{6.3}
\tilde{v}_2^k(y) \leq -2\log\tilde{L}_k-N_k, \qquad \mbox{for } |y|=\tilde{L}_k,
\end{equation}
for some $N_k\to+\infty$.

Since on $\partial B_{\tilde{L}_k}$ both components $\tilde{v}_1^k, \tilde{v}_2^k$ have fast decay we apply the argument of Remark \ref{rem-fast} and compute the Pohozaev identity. Since (\ref{6.3}) holds we get also
$$
\sigma_2^k(\tau_k \tilde L_k)=4\bar n + o(1),
$$
for some $\bar n\in\mathbb{N}$. Putting together the Pohozaev identity with the fact that both local masses are a small perturbation of a multiple of $4$ we conclude $(\sigma_1^k(\tau_k\tilde{L}_k),\sigma_2^k(\tau_k\tilde{L}_k))$ is a $o(1)$ perturbation of one of the two following types:
\begin{equation}\label{alt}
\bigr(2\tilde m(\tilde m+1),2\tilde m(\tilde m-1)\bigr)\qquad \mbox{or} \qquad \bigr(2\tilde m(\tilde m-1),2\tilde m(\tilde m+1)\bigr),
\end{equation}
for some $\tilde m\in\mathbb{N}.$ Without loss of generality we assume the former happens. As in the proof of Proposition \ref{pr4.1} we have
\begin{align*}
\begin{split}
&\overline{u}_k(\tau_k\tilde{L}_k)\leq-2\log(\tau_k\tilde{L}_k)-N_k, \\
&\frac{d}{dr}\overline{u}_k<-\frac{2+\e}{r},
\end{split}
\qquad \mbox{for } r=\tau_k\tilde{L}_k,
\end{align*}
for some $\e>0$. Now, following the steps in the proof of Proposition \ref{pr4.1}, as $r$ grows from $\tau_k\tilde{L}_k$ to $\tau_kL_k$, the following three cases may happen:

\

\noindent{\bf Case 1.} Both $u_k$ and $-u_k$ have fast decay up to $|x|=\tau_kL_k$:
\begin{align*}
|u_k(x)|\leq -2\log|x| -N_k, \qquad \tau_k\tilde{L}_k\leq|x|\leq\tau_kL_k,
\end{align*}
for some $N_k\to+\infty$. In this case, by Lemma \ref{le4.1} we have
\begin{equation*}
\sigma_i^k(\tau_kL_k)=\sigma_i^k(\tau_k\tilde{L}_k)+o(1), \quad i=1,2.
\end{equation*}
\medskip

\noindent{\bf Case 2.} There exists $L_{1,k}\in(\tilde{L}_k,L_k)$, $L_{1,k}=o(1)L_k $ such that
\begin{align*}
-u_k(x)&\geq -2\log L_{1,k}-C, \qquad \mbox{for } |x|=\tau_kL_{1,k}.
\end{align*}
By the argument of Lemma \ref{le4.2} we can find a suitable $L_{2,k}\geq L_{1,k}$ such that
\begin{align*}
|u_k(x)|&\leq-2\log L_{2,k}-N_k, \qquad \mbox{for } |x|=\tau_kL_{2,k},
\end{align*}
for some $N_k\to+\infty$ and $\bigr(\sigma_1^k(\tau_kL_{2,k}),\sigma_2^k(\tau_kL_{2,k})\bigr)$ is a $o(1)$ perturbation of
$$
\bigr(2\bar m(\bar m-1),2\bar m(\bar m+1)\bigr),
$$
for some $\bar m\in\mathbb{N}$.
\medskip

\noindent{\bf Case 3.} $-u_k$ has slow decay for $|x|=\tau_kL_k$, i.e.
\begin{align*}
-u_k(x)\geq-2\log\tau_kL_k-C, \qquad |x|=\tau_kL_k,
\end{align*}
for some $C>0$ and
$$
\sigma_1^k(\tau_kL_k)=\sigma_1^k(\tau_k\tilde{L}_k)+o(1)=4\bar n +o(1).
$$
Moreover, on $\partial B_{\tau_kL_k}(0)$, $u_k$ is still the fast decaying component.

\

The only region we have still to analyze is $B_{\tau_kL_k}(0)\setminus B_{\tau_kL_{2,k}}(0)$ when the second case above happens. However, the argument here is the same as before. At the end, in any case on $\partial B_{\tau_kL_k}(0)$ at least one of the two component $u_k,-u_k$ has fast decay and its energy is a small perturbation of a multiple of $4$.

\

Finally, we have to combine the groups. The procedure is very similar to the combination of bubbling disks as we have done before. For example, we start by considering groups which are close to each other: take $B_{\e_k}(0)$ for some $\e_k\rightarrow0$ such that all the groups in $B_{2\e_k}(0)$, say $G_0,G_1,\cdots,G_t$, (namely $(\Sigma_k\setminus(\cup_{i=0}^tG_i))\cap B(0,2\e_k)=\emptyset$) satisfy
\begin{align*}
&\mbox{dist}(G_i,G_j) \sim \mbox{dist}(G_l,G_q), \qquad \forall i\neq j, l\neq q, \\
&\mbox{dist}(G_i,G_j)=o(1)\e_k, \qquad \forall i,j=0,\cdots,t,~i\neq j.
\end{align*}
The second property implies that the groups outside $B_{2\e_k}(0)$ are far away from the groups inside the ball. By the above assumptions, letting $\e_{1,k}=dist(G_0,G_j)$, for some $j\in\{1,\dots,t\}$ we have that all $G_0,\cdots,G_t$ are in $B_{N\e_{1,k}}(0)$ for some $N>0$ independent of $k$. Without loss of generality let $u_k$ be the fast decaying component on $\partial B_{N\e_{1,k}}(0).$ Then we have
\begin{equation*}
\sigma_1^k(N\e_{1,k})=\sigma_1^k(\tau_kL_k)+4\hat m+o(1),
\end{equation*}
for some $\hat m\in\mathbb{N}$ because by Lemma \ref{le2.1} $u_k$ is also a fast decaying component for $G_0,\cdots,G_t$.

\

Now, as before we have three possible cases. If $-u_k$ also has fast decay on $\partial B_{N\e_{1,k}}(0),$  $\sigma_2^k(N\e_{1,k})$ is also a small perturbation of a multiple of $4$ and we get the quantization as in \eqref{alt}.

If instead
\begin{align*}
-u_k(x)\geq-2\log N\e_{1,k}-C, \qquad |x|=N\e_{1,k},
\end{align*}
then as before we can find $\e_{2,k}$ in $(N\e_{1,k},\e_k)$ such that
\begin{align*}
|u_k(x)|\leq -2\log\e_{2,k}-N_k, \qquad |x|=\e_{2,k},
\end{align*}
for some $N_k\rightarrow\infty.$ Moreover
\begin{align*}
\sigma_1^k(\e_{2,k})=\sigma_1^k(N\e_{1,k})+o(1).
\end{align*}
Thus, by the usual argument we get the quantization as in \eqref{alt}.

The last possibility is
\begin{align*}
\sigma_1^k(\e_k)=\sigma_1^k(N\e_{1,k})+o(1)=\sigma_1^k(\tau_kL_k)+4\hat m+o(1)
\end{align*}
and
\begin{align*}
-u_k(x)\geq-\log\e_k-C, \qquad |x|=\e_k,
\end{align*}
for some $C>0$. In this case $u_k$ is the fast decaying component on $\partial B_{\e_k}(0).$

Observe that at the end, in any case on $\partial B_{\e_k}(0)$ at least one of the two component $u_k,-u_k$ has fast decay and its energy is a small perturbation of a multiple of $4$.

\

With this argument we continue to include groups further away from $G_0$. Since by construction we have only finite blow-up disks this procedure only needs to be applied finite times. Finally, reasoning as in Lemma \ref{le3.1} we can take $s_k\rightarrow0$ such that $\Sigma_k\subset B_{s_k}(0)$ and both component $u_k, -u_k$ have fast decay on $\partial B_{s_k}(0)$:
\begin{align*}
|u_k(x)|\leq -2\log s_k-N_k, \qquad \mbox{for } |x|=s_k,
\end{align*}
for some $N_k\rightarrow\infty.$ Therefore we have that $(\sigma_1^k(s_k),\sigma_2^k(s_k)$ is a $o(1)$ perturbation of one of the two following types:
$$
\bigr(2m(m+1),2m(m-1)\bigr)\qquad \mbox{or} \qquad \bigr(2m(m-1),2m(m+1)\bigr),
$$
for some $m\in\mathbb{N}$. On the other hand, notice that by definition
\begin{align*}
\sigma_i=\lim_{k\rightarrow\infty}\lim_{s_k\rightarrow0}\sigma_i^k(s_k), \qquad i=1,2.
\end{align*}
It follows that $\sigma_1, \sigma_2$ satisfy the quantization property of Theorem \ref{th} and the proof is completed.
\begin{flushright}
$\square$
\end{flushright}


\vspace{1cm}
\begin{thebibliography}{99}


\bibitem{ba} A. Bahri and J.M. Coron, The scalar curvature problem on the standard three dimensional sphere. \emph{J. Funct. Anal.} 95 (1991), no. 1, 106-172.

\bibitem{bjmr} L. Battaglia, A. Jevnikar, A. Malchiodi and D. Ruiz, A general existence result for the Toda system on compact surfaces. To appear in {\em Adv. Math}.

\bibitem{bat} L. Battaglia and G. Mancini, A note on compactness properties of singular Toda systems. To appear in \emph{Atti Accad. Naz. Lincei Rend. Lincei Mat. Appl.}

\bibitem{bc1} H. Brezis and J.M. Coron, Multiple solutions of H-systems and Rellich's conjecture. {\em Commun. Pure Appl. Math.} 37 (1984), 149-187.

\bibitem{bm} H. Brezis and F. Merle, Uniform estimates and blow-up behavior for solutions of $-\Delta u=V(x)\,e^u$ in two dimensions.
{\em Comm. Partial Differential Equation} 16 (1991), 1223-1254.

\bibitem{clmp} E. Caglioti, P.L. Lions, C. Marchioro and M. Pulvirenti, A special class of stationary flows for two-dimensional Euler equations: a statistical mechanics description. \emph{Comm. Math. Phys.} 143 (1992), no. 3, 501-525.

\bibitem{cha}  S.Y.A. Chang, M.J. Gursky and P.C. Yang, The scalar curvature equation on $2$- and $3$- spheres. \emph{Calc. Var. and Partial Diff. Eq.} 1 (1993), no. 2, 205-229.

\bibitem{cha2} S.Y.A. Chang and P.C. Yang, Prescribing Gaussian curvature on $S^2$. \emph{Acta Math.} 159 (1987), no. 3-4, 215-259.

\bibitem{cl2} C.C. Chen and C.S. Lin, Estimate of the conformal scalar curvature equation via the method of
moving planes. II. {\em J. Differential Geom.} 49 (1998), no. 1, 115-178.

\bibitem{cl} W.X. Chen and C.M. Li, Classification of solutions of some nonlinear elliptic equations. {\em Duke
Math. J.} 63 (1991), no. 3, 615-622.

\bibitem{c} A.J. Chorin, Vorticity and Turbulence, Springer, New York, 1994.

\bibitem{ew} P. Esposito and J. Wei, Non-simple blow-up solutions for the Neumann two-dimensional sinh-Gordon equation. \emph{Calc. Var. PDEs} 34 (2009), no. 3, 341-375.

\bibitem{gp} M. Grossi and A. Pistoia, Multiple blow-up phenomena for the sinh-Poisson equation. {\em Arch. Ration. Mech. Anal.} 209 (2013), no. 1, 287-320.

\bibitem{han} Z.C. Han, Asymtotic approach to singular solutions for nonlinear elliptic equations involving critical Sobolev exponent. \emph{Annales de l'I.H.P., section C} 8 (1991), no. 2, 159-174.

\bibitem{jev} A. Jevnikar, An existence result for the mean field equation on compact surfaces in a doubly supercritical regime. \emph{Proc. Royal Soc. Edinb.} 143 (2013), no. 5, 1021-1045.

\bibitem{jev2} A. Jevnikar, Multiplicity results for the mean field equation on compact surfaces. To appear in \emph{Adv. Nonlinear Stud.}

\bibitem{jev3} A. Jevnikar. New existence results for the mean field equation on compact surfaces via degree theory. To appear in \emph{Rend. Semin. Mat. Univ. Padova}.

\bibitem{jo} J. Jost, Two-dimensional geometric variational problems. Pure and Applied Mathematics, Wiley, Chichester, 1991.

\bibitem{jw} J. Jost and G.F. Wang, Analytic aspects of the Toda system. I. A Moser-Trudinger inequality. {\em Comm. Pure Appl. Math.} 54 (2001), no. 11, 1289-1319.

\bibitem{jwyz} J. Jost, G. Wang, D. Ye and C. Zhou, The blow-up analysis of solutions of the elliptic sinh-Gordon equation. {\em Calc. Var. PDEs} 31 (2008), 263-276.

\bibitem{jm} G. Joyce and D. Montgomery, Negative temperature states for a two dimensional guiding center plasma. {\em J. Plasma Phys.} 10 (1973), 107-121.

\bibitem{ki} M.K.H. Kiessling, Statistical mechanics of classical particles with logarithmic interactions. \emph{Comm. Pure Appl. Math.} 46 (1993), no. 1, 27-56.

\bibitem{ly} Y.Y. Li, Prescribing scalar curvature on $S^n$ and related problems. I. {\em J. Differential Equations}
120 (1995), no. 2, 319-410.

\bibitem{li} Y.Y. Li, Harnack type inequality: the method of moving planes. \emph{Comm. Math. Phys.} 200 (1999), no. 2, 421-444.

\bibitem{li-sha} Y.Y. Li and I. Shafrir, Blow-up analysis for solutions of $-\Delta u = V\, e^u$ in dimension two.  \emph{Indiana Univ. Math. J.} 43 (1994), no. 4, 1255-1270.

\bibitem{lwy} C.S. Lin, J.C. Wei and W. Yang, Degree counting and shadow system for SU(3) Toda system: one bubbling, preprint, 2014, arXiv http://arxiv.org/abs/1408.5802.

\bibitem{lwz0} C.S. Lin, J.C. Wei and L. Zhang,  Classification of blowup limits for $SU(3)$ singular Toda systems. {\em Anal. PDE} 8 (2015), no. 4, 807-837.

\bibitem{l} P.L. Lions, On Euler Equations and Statistical Physics, Scuola Normale Superiore, Pisa, 1997.

\bibitem{mal} A. Malchiodi, Topological methods for an elliptic equation with exponential nonlinearities. \emph{Discrete Contin. Dyn. Syst.} 21 (2008), no. 1, 277-294.

\bibitem{mp} C. Marchioro and M. Pulvirenti, Mathematical Theory of Incompressible Nonviscous Fluids, Springer, New York, 1994.

\bibitem{n} P.K. Newton, The N-Vortex Problem: Analytical Techniques, Springer, New York, 2001.

\bibitem{os} H. Ohtsuka and T. Suzuki, Mean field equation for the equilibrium turbulence and a related functional inequality. {\em Adv. Differ. Equ.} 11 (2006), 281-304.

\bibitem{os1} H. Ohtsuka and T. Suzuki, A blowup analysis of the mean field equation for arbitrarily signed vortices. \emph{Self-similar solutions of nonlinear PDE} 74 (2006), 185-197.

\bibitem{o} L. Onsager, Statistical hydrodynamics. {\em Nuovo Cimento} 6 (1949), 279-287.

\bibitem{pa} T.H. Parker, Bubble tree convergence for harmonic maps. \emph{J. Differ. Geom.} 44 (1996), 595-633.

\bibitem{pr} A. Pistoia and T. Ricciardi, Concentrating solutions for a Liouville type equation
with variable intensities in 2D-turbulence, Preprint. ArXiv http://arxiv.org/pdf/1507.01449v1.pdf.

\bibitem{rt} T. Ricciardi and R. Takahashi, Blow-up behavior for a degenerate elliptic sinh-Poisson equation with
variable intensities, Preprint. ArXiv http://arxiv.org/pdf/1507.01449.pdf.

\bibitem{rtzz} T. Ricciardi, R. Takahashi, G. Zecca and X. Zhang, On the existence and blow-up of solutions for a mean field equation with variable intensities, Preprint. ArXiv http://arxiv.org/pdf/1509.05204.pdf.

\bibitem{su} J. Sacks and K. Uhlenbeck, The existence of minimal immersions of 2-spheres. {\em Ann. Math.} 113 (1981), 1-24.

\bibitem{s3} R. Schoen, Stanford Classnotes.

\bibitem{s}R. Schoen and D. Zhang, Prescribed scalar curvature on the n-sphere. \emph{Calc. Var.} 4 (1996), no. 1, 1-25.

\bibitem{s1} K. Steffen, On the nonuniqueness of surfaces with prescribed constant mean curvature spanning a given contour. {\em Arch. Rat. Mech. Anal.} 94, (1986) 101-122.

\bibitem{s2} M. Struwe, Nonuniqueness in the Plateau problem for surfaces of constant mean curvature. {\em Arch. Rat. Mech. Anal.} 93, (1986) 135-157.

\bibitem{tar} G. Tarantello, Analytical, geometrical and topological aspects of a class of mean field equations on surfaces. \emph{Discrete Contin. Dyn. Syst.} 28 (2010), no. 3, 931-973.

\bibitem{w1}
H.C. Wente, Large solutions to the volume constrained Plateau problem. \emph{Arch. Rational Mech.
Anal.} 75 (1980/81), no. 1, 59-77.

\bibitem{w2}
H.C. Wente, Counterexample to a conjecture of H. Hopf. \emph{Pacific J. Math.} 121 (1986), no. 1, 193-243.

\bibitem{zh} C. Zhou, Existence result for mean field equation of the equilibrium turbulence in the super critical case. \emph{Commun. Contemp. Math.} 13 (2011), no 4, 659-673.




\end{thebibliography}
\end{document}